\documentclass[letterpaper, 10 pt, conference, onecolumn]{ieeeconf}

\usepackage[utf8]{inputenc}    
\usepackage{amssymb}
\usepackage{amsfonts}
\usepackage{amsmath}
{\newtheorem{assumption}{Assumption}}
\usepackage{graphicx}
\usepackage{mathtools}
\usepackage{xfrac}
\usepackage{nicefrac}
\usepackage{mathrsfs}
\usepackage{breqn}
\usepackage{dsfont}
\usepackage{wasysym}
\usepackage{tabularx}
\usepackage{booktabs}
\usepackage{multirow}
\usepackage{relsize}
\usepackage{times}
\usepackage{url}
\usepackage{adjustbox}
\usepackage{boldline}
\usepackage{algorithm}
\usepackage[T1]{fontenc}
\usepackage[dvips]{epsfig}
\usepackage{epstopdf}
\usepackage{color}
\usepackage[percent]{overpic}
\DeclareGraphicsExtensions{.ps,.eps,.pdf}

\newfloat{procedure}{htbp}{loa}
\floatname{procedure}{Procedure}

\IEEEoverridecommandlockouts

\newtheorem{remark}{Remark}

\newtheorem{lemma}{Lemma}

\newcommand{\tphip}{\tilde{\varphi}_p}
\newcommand{\ulambp}{\underline{\lambda}_p}
\newcommand{\tphikp}{\tilde{\varphi}_p(k-p)}
\newcommand{\taulkp}{\hat{\tau}_p(\tilde{\varphi}_p(k-p),\theta_p)}
\newcommand{\taugp}{\bar{\hat{\tau}}_p(\theta_p)}
\newcommand{\hthetap}{\hat{\theta}_p}
\newcommand{\taugph}{\bar{\hat{\tau}}_p(\hat{\theta}_p)}

\begin{document}
\date{}

\title{Data-driven filtering for linear systems using\\Set Membership multistep predictors}

\author{Marco~Lauricella  and 
	Lorenzo~Fagiano%
	\thanks{The authors are with the Dipartimento di Elettronica, Informazione e Bioingegneria, Politecnico di Milano, Piazza Leonardo da Vinci
32, 20133 Milano, Italy (e-mail: marco.lauricella@polimi.it; lorenzo.fagiano@polimi.it). \textit{Corresponding author: Marco Lauricella}.\newline This is the extended version of an accepted paper submitted to the 2020 IEEE Conference on Decision and Control.}}

\maketitle

\begin{abstract}
This paper presents a novel data-driven, direct filtering approach for unknown linear time-invariant systems affected by unknown-but-bounded measurement noise. The proposed technique combines independent multistep prediction models, identified resorting to the Set Membership framework, to refine a set that is guaranteed to contain the true system output. The filtered output is then computed as the central value in such a set. By doing so, the method achieves an accurate output filtering and provides tight and minimal error bounds with respect to the true system output. To attain these results, the online solution of linear programs is required. A modified filtering approach with lower online computational cost is also presented, obtained  by moving the solution of the optimization problems to an offline preliminary phase, at the cost of larger accuracy bounds. The performance of the proposed approaches are evaluated and compared with those of standard model-based filtering techniques in a numerical example. 
\end{abstract}

\section{Introduction}\label{s:intro}
In this paper, we address the problem of output filtering for the case of linear time-invariant systems subject to unknown-but-bounded measurement disturbances. Our goal is to obtain an accurate filtering of the output from the noise, and, at the same time, to provide tight bounds on the error between the filtered output and the true one. 
The problem of process variables filtering has been widely studied over the years, and the related literature includes a large number of contributions, see, e.g., \cite{gelb1974applied,anderson2012optimal,saberi2007filtering} and the references therein. Among them, the most famous filtering approaches stem from the seminal works of Kolmogorov-Wiener \cite{kolmogorov1941stationary,wiener1949smoothing}, and of Kalman-Bucy \cite{kalman1961new}. These works, as a vast majority of the filtering methods proposed during the years, mainly focus on the case of perfectly known system model, and stochastic disturbances with known probabilistic properties, for which they achieve optimal filtering properties. Other approaches assume that the accuracy is measured by suitable noise to error norm gains, where optimal performance are achieved by $\mathcal{H}_2$, $\ell_1$, $\mathcal{H}_{\infty}$ filters, see e.g. \cite{Geromel2001,Sun2005,Nagpal1991,Abedor1996,Voulgaris1995,Colaneri2002,Shaked1992Hinf}. A different filtering methodology is also given by Moving Horizon Estimators, see e.g. \cite{alessandri2008moving,rao2003constrained}. Most of these filtering approaches are mainly based on the assumption of exact model knowledge as well. Since this is rarely the case in real world applications, where the system model is often unknown, a two step procedure is commonly adopted, where a model of the system at hand is identified from a noise-affected set of data collected from an experiment performed on the system, which is then used to define the filter. In this way, no optimality properties can be guaranteed anymore.

Here, we resort to Set Membership (SM) identification methods to design a data-based, direct (i.e., without estimating a model of the plant) output filtering approach, able to obtain a `small' estimation error, measured by suitable guaranteed accuracy bounds. 
SM identification methods usually operate under the unknown-but-bounded disturbance and uncertainty framework, see e.g., \cite{milanese2013bounding,walter1990estimation,garulli1999robustness}, and can provide strong results in terms of optimality of the accuracy bounds in a worst-case sense, while requiring rather mild assumptions on the system and the disturbances \cite{milanese2011unified}.

In the SM literature, there are several contributions addressing the problem of data-driven filter design, e.g., \cite{milanese2010direct,yang2009set,garulli1997conditional}, where a filtered version of the variables of interest is provided, along with guaranteed bounds on their accuracy. Inspired by the results of a previous work (\cite{lauricella2020set}), where the theoretical properties of the error bounds of multistep and onestep iterated models are investigated, here, we propose a novel filtering algorithm, based on the combined use of independent multistep prediction models, that is able to achieve tight accuracy bounds. The main idea is to derive a set containing the true system output at a given time, as the intersection of the accuracy regions provided by several SM multistep predictors, each one with a different prediction horizon. Then, the wanted filtered output is computed as the center of such a set. This quantity is locally optimal, in a worst-case error sense, i.e. it attains the minimal achievable error bounds under the considered assumptions. This is possible thanks to the use of independent multistep predictors, allowing one to achieve good prediction accuracy for different steps ahead, see e.g., \cite{shook1991identification,haber2003long,farina2011simulation,lauri2010pls}, and to derive less conservative bounds on the prediction accuracy, when compared to a single prediction model iterated in simulation \cite{shook1992control,sjoberg1995nonlinear}. Moreover, independent $p$-steps ahead models also allow one to achieve the theoretical minimum $p$-steps ahead prediction error bound, as defined under the SM framework, see \cite{terzi2019learning}. By combining the output of different multistep predictors, all attaining their theoretical minimum error, we can further refine the set containing the system output, thus obtaining even tighter accuracy bounds.\\
The approach requires the online solution of linear programs (LPs), which can be afforded in many real world applications of interest, thanks to the computational power available nowadays. For  cases where the available computational power is limited, and/or the sampling time is too short for the online solution of optimization problems, we also present a second filtering approach based on the same principle, which uses global accuracy bounds computed in a preliminary offline identification phase, at the cost of higher conservativeness.

The filtering approaches proposed in this paper are finally showcased on a numerical example, where their performance are compared to those of a standard robust Kalman filter based on the Riccati equation approach.

\section{Problem statement and preliminary results}\label{s:problem_statement}

\subsection{Problem statement}\label{SS:prob_for}
Consider a discrete time, linear time invariant system, with input $u(k)\in\mathbb{R}^m$, and output $z(k)\in\mathbb{R}$, where $k\in\mathbb{Z}$ denotes the discrete time variable, modeled using an input-output auto-regressive with exogenous input (ARX) representation of order $o$ (see \cite{lauricella2020set} for details regarding the order choice)
\begin{equation}\label{eq:system_ARX_1s}
z(k+1)=\psi(k)^T\theta^0,
\end{equation}
where $\theta^0$ is the vector of the real system parameters, and the regressor $\psi(k)$ is defined as
\begin{equation*}\label{eq:system_regressor_def}
\begin{aligned}
\psi(k)=&\left[ Z_o^T(k) \; U_o^T(k) \right]^T\in \mathbb{R}^{o+mo},\\
Z_o(k)=&\left[ z(k) \; z(k-1) \; \hdots \; z(k-o+1) \right]^T\in \mathbb{R}^{o},\\
U_o(k)=&\left[u(k)^T \; u(k-1)^T \hdots\; u(k-o+1)^T \right]^T\in \mathbb{R}^{mo}.
\end{aligned}
\end{equation*}
The output measurement $y(k)\in\mathbb{R}$ is affected by an additive noise $d(k)\in\mathbb{R}$:
\begin{equation*}\label{eq:output_meas}
y(k)=z(k)+d(k).
\end{equation*}
For the sake of notational simplicity, and without loss of generality, here we consider a multiple input, single output system. The proposed approach can be easily generalized to the multiple output case, by applying it to one output at the time. \\
Let us make the following assumption on the system at hand:
\begin{assumption}\label{as:bounded_dist}
The measurement noise and the system input are bounded. In particular:
\begin{itemize}
\item $|d(k)| \leq \bar{d}_0,\; \forall k\in\mathbb{Z}, \; \bar{d}_0 \in \mathbb{R}, \; \bar{d}_0>0.$
\item $u(k)\in\mathbb{U}\subset \mathbb{R}^m, \; \forall k \in \mathbb{Z}, \; \mathbb{U} \; \text{compact}.$
\end{itemize}
\end{assumption}
\noindent Assumption \ref{as:bounded_dist} is common for system identification problems based on the framework of unknown-but-bounded disturbance, and it is valid for many practical applications as well.
Here, we consider independent $p$-steps ahead prediction models in the ARX form
\begin{equation}\label{eq:pred_formulation}
\hat{z}(k|k-p)=\varphi_p(k-p)^T\theta_p,
\end{equation}
where $\theta_p$ is the vector of the predictor parameters to be identified, and the noise-affected regressor $\varphi_p(k-p)$ is defined as
\begin{equation}\label{eq:pred_regressor_def}
\begin{aligned}
\varphi_p(k-p)=&\left[ Y_o^T(k-p) \; U_{p,o}^T(k-p) \right]^T\in \mathbb{R}^{o+m(o+p-1)},\\
Y_o(k-p)=&\left[ y(k-p) \; y(k-p-1) \; \hdots \; y(k-p-o+1) \right]^T\in \mathbb{R}^{o},\\
U_{p,o}(k-p)=&\left[u(k-1)^T \; \hdots \; u(k-p)^T \; \hdots \; u(k-p-o+1)^T \right]^T\in \mathbb{R}^{m(o+p-1)}.
\end{aligned}
\end{equation}
Notice that a $p$-steps ahead model of the system can be readily obtained by recursion of \eqref{eq:system_ARX_1s}, leading to
\begin{equation*}
z(k|k-p)=\psi_p(k-p)^T\theta_p^0,
\end{equation*}
where $\psi_p(k-p)$ is the noise-free counterpart of $\varphi_p(k-p)$, obtained substituting $Y_o(k)$ with $Z_o(k)$ in \eqref{eq:pred_regressor_def}, and the parameter values in $\theta_p^0$ are polynomial functions of the entries of $\theta^0$.\\
Let us assume that an experiment is performed on the system at hand to collect a finite number of sampled data $(\tilde{y}(k),\tilde{u}(k))$ to be used for the identification task. Here, $\tilde{\cdot}$ is used to denote a sampled and stored data point of a variable. For each prediction step $p$, these data form the following set, composed of $N$ sampled regressors and $N$ corresponding $p$-steps ahead output measurements:
\begin{equation}\label{eq:samp_dataset_arx}
\tilde{\mathscr{V}}_p^N \doteq \left\{ \tilde{v}_p(k)=\begin{bmatrix} \tphikp \\ \tilde{y}(k) \end{bmatrix}, \, k=1,\hdots,N \right\}, 
\end{equation}
where $\tilde{\mathscr{V}}_p^N \subset \mathbb{R}^{1+o+m(o+p-1)}$. For simplicity and without loss of generality, we consider that the number $N$ of data points is the same for any considered value of $p$. Moreover, as customary for identification problems based on sampled data, we consider the following assumption on the available dataset:
\begin{assumption}\label{as:info_data}
As $N\to\infty$, the sampled dataset $\tilde{\mathscr{V}}_p^N$ is such that the set of all the system trajectories of interest is densely covered, i.e., the input $u$ is persistently exciting and the disturbance $d$ is bound-exploring.
\end{assumption}

Prediction models \eqref{eq:pred_formulation} are also known as multistep predictors, as they are designed to directly provide the output prediction $p$-steps ahead, and do not require the integration of an underlying simulation model. Note that, here, we consider prediction models having the same autoregressive order of \eqref{eq:system_ARX_1s}. In case the order $o$ is not known, it can be estimated resorting to the procedure proposed in \cite{lauricella2020set}. Moreover, in the remainder of this paper, we assume that either the disturbance bound $\bar{d}$ is known, or that its estimate, obtained as described in \cite{lauricella2020set}, is available.

The problem addressed in this paper is the following: given the data set \eqref{eq:samp_dataset_arx}, derive a filtering algorithm that returns an estimate $\hat{z}(k)\approx z(k)$ of the system output, together with guaranteed bounds on the error $|\hat{z}(k)- z(k)|$.

\subsection{Preliminary results}\label{SS:preliminary results}
The error between the true system output $z(k)$ of \eqref{eq:system_ARX_1s} and the $p$-steps ahead prediction $\hat{z}(k|k-p)$, given by \eqref{eq:pred_formulation}, originates from two sources:  the difference between the real parameter values $\theta^0$ and the predictor ones, and the noise affecting the measurements used in the regressor $\tphikp$. Under Assumption \ref{as:bounded_dist}, this $p$-steps ahead error can be upper bounded by a worst-case error bound $\ulambp$, which can be estimated resorting to the SM framework, see \cite{lauricella2020set}, as
\begin{equation}\label{eq:lambda_calc}
\begin{array}{c}
\ulambp = \alpha \cdot \min\limits_{\theta_p,\lambda\in \mathbb{R}^+} \lambda \\
\text{subject to} \\
\left\vert \tilde{y}-\tphip^T \theta_p \right\vert \leq \lambda +\bar{d}, \;\; \forall \left(\tphip,\tilde{y} \right): \begin{bmatrix} \tphip \\ \tilde{y} \end{bmatrix} \in \tilde{\mathscr{V}}_p^N,
\end{array}
\end{equation}
with $\alpha>1$. For a fixed value of $p$, $\ulambp$ represents the global error bound related to all possible $p$-steps ahead predictors of the form of \eqref{eq:pred_formulation}, and it is used to define the Feasible Parameter Set (FPS):
\begin{equation}\label{eq:FPS_gen_def}
\Theta_p \doteq \bigg\{ \theta_p : \left\vert \tilde{y} - \tphip^T \theta_p \right\vert \leq \ulambp+\bar{d}, \; \forall \left(\tphip,\tilde{y} \right) : \begin{bmatrix} \tphip \\ \tilde{y} \end{bmatrix} \in \tilde{\mathscr{V}}_p^N \bigg\}.
\end{equation}
The FPS $\Theta_p$ is the set of all possible $p$-steps ahead predictor parameters that are consistent with the sampled data and the disturbance bound  $\bar{d}$.
Under Assumption \ref{as:info_data}, $\Theta_p$ is a polytope having at most $2N$ facets, represented in \eqref{eq:FPS_gen_def} with an inequality description. If $\Theta_p$ happens to be unbounded, it is and indication that the data collected from the system are not informative enough and/or that $N$ is not big enough, invalidating Assumption \ref{as:info_data}, thus new data should be acquired. Moreover, the bound $\ulambp$ includes a factor $\alpha>1$ (see \eqref{eq:lambda_calc}) to account for the uncertainty due to the usage of a finite dataset. More details regarding the error bounds, the FPS, their theoretical properties, and the usage of scaling parameters can be found in \cite{lauricella2020set}.
\section{Multistep filtering with local bounds}\label{s:filtering}
Under the SM framework, it is possible to associate guaranteed accuracy bounds to multistep prediction models in the form of \eqref{eq:pred_formulation}, resorting to the bounds $\ulambp$, and the FPSs $\Theta_p$ defined in Section \ref{s:problem_statement}, see e.g., \cite{lauricella2020set}. 

The guaranteed accuracy bound $\taulkp$ is defined, for a $p$-steps ahead predictor with parameters $\theta_p$ and for a given regressor $\tilde{\varphi}_p(k-p)$, as
\begin{equation}\label{eq:tau_local_def}
\taulkp=\gamma \left(\max_{\theta \in \Theta_p} \left\vert \tilde{\varphi}_p(k-p)^T(\theta-\theta_p) \right\vert \right) + \ulambp,
\end{equation}
with $\gamma>1$. Here, $\gamma$ has the same role of the scaling parameter $\alpha$ in \eqref{eq:lambda_calc}, thus accounting for the uncertainty due to the usage of a finite dataset. The bound \eqref{eq:tau_local_def} is named ``local'', since it pertains to a specific regressor value $\tilde{\varphi}_p(k-p)$. Under Assumptions \ref{as:bounded_dist} and \ref{as:info_data}, it holds that, by construction
\begin{equation}\label{eq:z_under_bound_tau}
|z(k)-\tilde{\varphi}_p(k-p)^T\theta_p | \leq \taulkp,
\end{equation}
i.e., the system output is guaranteed to lie inside a set defined by an interval centered at the $p$-steps ahead prediction $\hat{z}(k|k-p)=\tilde{\varphi}_p(k-p)^T\theta_p$, with amplitude equal to the corresponding accuracy bound.\\
Resorting to this property, the novel data-driven filtering approach we propose is able to refine the true system output uncertainty range. This result is obtained by intersecting the uncertainty intervals of the outputs given by a group of $\bar{p}$ multistep models  \eqref{eq:pred_formulation}, all providing a prediction of the output $z(k)$, ranging from $\hat{z}(k|k-1)$ to $\hat{z}(k|k-\bar{p})$. Thus, the set $Z_{\bar{p}}(k)$ containing the true output at time $k$ is derived as:
\begin{equation}\label{eq:z_uncert_region_def_gen}
\begin{aligned}
z(k)\in Z_{\bar{p}}(k)=\Big\{ \bar{z}&: \, \tilde{\varphi}_p(k-p)^T\theta_p-\hat{\tau}_p(\tilde{\varphi}_p(k-p),\theta_p) \leq\bar{z} \leq \tilde{\varphi}_p(k-p)^T\theta_p+\hat{\tau}_p(\tilde{\varphi}_p(k-p),\theta_p), \\
&\forall \theta_p\in\Theta_p,\;\forall p=1,\hdots,\bar{p} \Big\}.
\end{aligned}
\end{equation}
The boundaries of $Z_{\bar{p}}(k)$ are:
\begin{equation}
\begin{aligned}
z^{max}(k)&=\max_{z\in Z_{\bar{p}}(k)} z = \min_{p=1,\hdots,\bar{p}} \zeta_p^{max}, \\
z^{min}(k)&=\min_{z\in Z_{\bar{p}}(k)} z = \max_{p=1,\hdots,\bar{p}} \zeta_p^{min},
\end{aligned}
\end{equation}
where
\begin{subequations}\label{eq:zeta_i_minmax}
\begin{equation}\label{eq:zeta_i_max}
\zeta_p^{max}=\min_{\theta_p\in\Theta_p}\Big(\tilde{\varphi}_p(k-p)^T\theta_p+\taulkp\Big),
\end{equation}
\begin{equation}\label{eq:zeta_i_min}
\zeta_p^{min}=\max_{\theta_p\in\Theta_p}\Big(\tilde{\varphi}_p(k-p)^T\theta_p-\taulkp\Big).
\end{equation}
\end{subequations}
Note that problems \eqref{eq:zeta_i_max} and \eqref{eq:zeta_i_min} can be recast as 2 Linear Programs (LP), which requires the previous solution of 2 LPs. Take for example the original optimization problem:
\[
\zeta_p^{max}=\min\limits_{\theta_p\in\Theta_p} \Big( \tilde{\varphi}_p(k-p)^T\theta_p + \gamma\max\limits_{\theta\in\Theta_p} \left\vert \tilde{\varphi}_p(k-p)^T (\theta-\theta_p) \right\vert+\ulambp \Big).
\]
This can be formulated as a minimization problem, where the absolute value is removed:
\begin{equation}\label{eq:zeta_LP_probl_interm}
\begin{aligned}
\zeta_p^{max}=&\min_{\tau,\theta_p\in\Theta_p} \tilde{\varphi}_p(k-p)^T\theta_p+ \gamma\tau+\ulambp \\
&\text{subject to}\\
&\max_{\theta\in\Theta_p} \Big( \tilde{\varphi}_p(k-p)^T(\theta-\theta_p) \Big)\leq \tau \\
&\max_{\theta\in\Theta_p} \Big( -\tilde{\varphi}_p(k-p)^T(\theta-\theta_p) \Big)\leq \tau
\end{aligned}
\end{equation}
Then, the maximization problems in \eqref{eq:zeta_LP_probl_interm} can be solved as independent LPs, since
\[\max\limits_{\theta\in\Theta_p}\tilde{\varphi}_p(k-p)^T(\theta-\theta_p)=\max\limits_{\theta\in\Theta_p}(\tilde{\varphi}_p(k-p)^T\theta) - \tilde{\varphi}_p(k-p)^T\theta_p,\] leading to:
\begin{equation}\label{eq:zeta_i_max_LP}
\begin{aligned}
\zeta_p^{max}=&\min_{\tau,\theta_p\in\Theta_p} \Big( \tilde{\varphi}_p(k-p)^T\theta_p+ \gamma\tau\Big)+\ulambp\\
&\text{subject to}\\
&c_{1_{p_k}}-\tilde{\varphi}_p(k-p)^T\theta_p\leq\tau\\
&c_{2_{p_k}}+\tilde{\varphi}_p(k-p)^T\theta_p\leq\tau\\
\end{aligned}
\end{equation}
where $c_{1_{p_k}}=\max\limits_{\theta\in\Theta_p}\tilde{\varphi}_p(k-p)^T\theta$, and $c_{2_{p_k}}=\max\limits_{\theta\in\Theta_p}-\tilde{\varphi}_p(k-p)^T\theta$. Similarly, problem \eqref{eq:zeta_i_min} becomes
\begin{equation}\label{eq:zeta_i_min_LP}
\begin{aligned}
\zeta_p^{min}=-&\min_{\tau,\theta_p\in\Theta_p} \Big(-\tilde{\varphi}_p(k-p)^T\theta_p+ \gamma\tau\Big)-\ulambp\\
&\text{subject to}\\
&c_{1_{p_k}}-\tilde{\varphi}_p(k-p)^T\theta_p\leq\tau\\
&c_{2_{p_k}}+\tilde{\varphi}_p(k-p)^T\theta_p\leq\tau\\
\end{aligned}
\end{equation}
The filtered value of $z(k)$ is then obtained using a central algorithm on the set $Z_{\bar{p}}(k)$, since the center of the local uncertainty interval attains the minimal achievable local accuracy bound, in a worst-case error sense, see \cite{traub1988information}:
\begin{equation}\label{eq:z_filter_local}
\hat{z}_{f_{\bar{p}}}(k)\doteq \frac{1}{2}\Big(z^{max}(k)+z^{min}(k) \Big).
\end{equation}
In summary, the computation of \eqref{eq:z_filter_local} implies the solution of $4\bar{p}$ LPs having $2N$ constraints at each sampling time, and requires the offline computation of the error bounds $\ulambp$, for $p\in[1,\hdots,\bar{p}]$, obtained solving $\bar{p}$ LPs. Since the FPSs are defined using only data pertaining to the identification dataset, and they are not updated online, it is possible to reduce the computational effort needed to solve \eqref{eq:zeta_i_max_LP} and \eqref{eq:zeta_i_min_LP} by performing an offline redundant constraints removal procedure on the FPSs, as proposed also in \cite{lauricella2020set}. This produces a significant decrease in the number of constraints, reducing the computational effort needed for the online solution of the LPs. Examples of the complexity of $\Theta_p$ in terms of number of inequalities, and of computational times for the local approach, are provided in Section \ref{s:sim_results}.

Procedure \ref{p:MS_loc_filtering} summarizes the described local filtering algorithm.
\begin{procedure}
	\caption{Multistep filtering with local bounds}
	\label{p:MS_loc_filtering}
	\begin{enumerate}
		\item Carry out the offline estimation of $\ulambp$, and use it to define the FPSs for the considered steps $p\in[1,\,\bar{p}]$.
		\item For every time sample $k$, compute the online solution of the LPs \eqref{eq:zeta_i_max_LP} and \eqref{eq:zeta_i_min_LP}, for $p\in[1,\,\bar{p}]$.
		\item Use the central algorithm \eqref{eq:z_filter_local} to obtain the filtered system output, and compute its guaranteed accuracy bound for the given time sample.
	\end{enumerate}
\end{procedure}

The following result pertaining to the guaranteed accuracy of the local approach holds.
\begin{lemma}\label{lm:z_filt_prop_new}
	Considering the output filtering algorithm \eqref{eq:z_filter_local}, the following properties hold:
	\begin{enumerate}
		\item The guaranteed local accuracy bound $\tau_{f_{\bar{p}}}(k)$ is
		\begin{equation}\label{eq:z_filt_bound_new}
		|\hat{z}_{f_{\bar{p}}}(k)- z(k)|\leq\tau_{f_{\bar{p}}}(k)=\frac{1}{2}\left\vert z^{max}(k)-z^{min}(k) \right\vert.
		\end{equation}
		\item The accuracy bound \eqref{eq:z_filt_bound_new} is the smallest worst-case error bound that can be achieved by any predicted output value.
		\item The accuracy bound \eqref{eq:z_filt_bound_new} is smaller than any single bound $\taulkp$, $\forall p\leq\bar{p}$.
	\end{enumerate}
\end{lemma}
\begin{proof}
	Straightforward consequence of \eqref{eq:tau_local_def}-\eqref{eq:zeta_i_minmax}, \eqref{eq:z_filter_local}. 
\end{proof}
\begin{remark}\label{rm:pbar_choice}
The variable $\bar{p}$, which represents the maximum prediction horizon length used in the filtering algorithm, is a tunable variable, whose choice corresponds to a trade-off between the tightness of the accuracy bound, and the resulting computational effort (which increases with $\bar{p}$).
\end{remark}
\section{Multistep filtering with global bounds}\label{s:filter_global_tau}
The multistep filtering approach proposed in Section \ref{s:filtering} is able to achieve tight accuracy bounds, since it is based on the guaranteed local error bound \eqref{eq:tau_local_def}. To do so, the filter requires the online solution of $4\bar{p}$ LPs for each time sample. Even though this can be easily done in numerous real world applications, thanks to the computational power available nowadays, there are still a number of applications where the sampling time is too short, or the available computational power is limited, and the filtering approach based on local bounds would be inapplicable online. In all these cases, it is possible to define an output filtering algorithm similar to \eqref{eq:z_filter_local}, but based on a global version of the guaranteed accuracy bound. The main advantage is that fixed prediction models and their bounds are computed offline during the identification phase, together with $\ulambp$ and the FPSs, and no optimization problem needs to be solved online. On the other side, the global version of bound $\tau_p(\theta_p)$ is designed to hold for any possible regressor value, corresponding to any possible system trajectory, thus it is more conservative than its local counterpart, which is designed specifically for a given regressor instance. Moreover, the multistep predictors are identified only once, so their performance could clearly not be better than that of the locally optimal prediction models described in Section \ref{s:filtering}. To select such fixed predictors, a suitable optimality criterion needs to be defined. \\
Multistep prediction models can be identified resorting to several criteria, among which least-squares estimation is probably the most famous \cite{soderstrom1989system}. Since our goal here is to obtain tight uncertainty intervals for the unknown system output, even in the global case, it seems reasonable to seek for the multistep predictors that minimize the corresponding $p$-steps ahead global error bound, thus obtaining the theoretical minimum global uncertainty interval for the corresponding horizon length $p$. For given predictor parameters $\theta_p$, the corresponding guaranteed global accuracy bound is estimated as:
\begin{equation*}\label{eq:tau_hat_global_def}
\bar{\hat{\tau}}_p(\theta_p)=\bar{\gamma} \left( \max_{\tphip\in \tilde{\mathscr{V}}_p^N} \max_{\theta \in \Theta_p} \left\vert \tphip^T \left(\theta - \theta_p \right) \right\vert \right) + \ulambp, \; \bar{\gamma}>1.
\end{equation*}
Under assumption \ref{as:info_data}, it can be shown that the global bound $\taugp$ is such that \cite{terzi2019learning,lauricella2020set}:
\[
|z(k)-\tilde{\varphi}_p(k-p)^T\theta_p | \leq \bar{\hat{\tau}}_p(\theta_p),
\]
i.e. it holds for any regressor belonging to the set of system trajectories covered by data (compare with \eqref{eq:z_under_bound_tau}). Then, we search inside the FPS $\Theta_p$ for the parameter vector $\hat{\theta}_p$ that minimizes the global bound $\bar{\hat{\tau}}_p(\hat{\theta}_p)$, i.e. \[
\hat{\theta}_p=\text{arg}\min\limits_{\theta_p\in\Theta_p}\bar{\hat{\tau}}_p(\theta_p).\]
This corresponds to
\begin{equation}\label{eq:multistep_theta_id}
\hat{\theta}_p = \text{arg} \min\limits_{\theta_p\in\Theta_p}\; \max\limits_{k=p+1,\hdots,N}\; \max\limits_{\theta\in\Theta_p} \left\vert \tilde{\varphi}_p(k-p)^T (\theta-\theta_p) \right\vert 
\end{equation}
Problem \eqref{eq:multistep_theta_id} can be reformulated as an LP, which requires the previous solution of $2(N-p)$ LPs, see \cite{lauricella2020set}, as
\begin{equation}\label{eq:multistep_theta_LP}
\begin{array}{c}
\hat{\theta}_p = \text{arg} \min\limits_{\zeta,\theta_p\in\Theta_p} \zeta \\
\text{subject to} \\
c_{k_p}-\check{\varphi}_p(k)^T\theta_p\leq \zeta, \; k\in [p+1,\,N]\cup[N+p+1,\,2N], 
\end{array}
\end{equation}
where $c_{k_p}\doteq \max\limits_{\theta\in\Theta_p} \check{\varphi}_p(k)^T \theta$, and
\begin{equation}\nonumber
\check{\varphi}_p(k)=\begin{cases}
\begin{array}{lcl}
\tilde{\varphi}_p(k-p) & \text{if} & k\leq N \\
-\tilde{\varphi}_p(k-p-N) & \text{if} & k> N.
\end{array}
\end{cases}
\end{equation}
The multistep model identified using \eqref{eq:multistep_theta_LP} is tailored for a specific horizon length $p$, for which it achieves optimal prediction properties, in terms of guaranteed global error bound amplitude. Then, we can intersect again the uncertainty intervals defined by the global accuracy bounds, obtaining the set $Z^g_{\bar{p}}(k)$ such that:
\begin{equation*}\label{eq:Z_global_set}
\begin{aligned}
z(k)\in Z^g_{\bar{p}}(k)&=\Big\{ \bar{z}:\, \tilde{\varphi}_p(k-p)^T\hthetap-\bar{\hat{\tau}}_p(\hat{\theta}_p)\leq\bar{z}\leq\tilde{\varphi}_p(k-p)^T\hthetap+\bar{\hat{\tau}}_p(\hat{\theta}_p), \; \forall p=1,\hdots,\bar{p} \Big\}.
\end{aligned}
\end{equation*}
The filtered output value at time $k$ is then given by the center of the set $Z^g_{\bar{p}}$:
\begin{equation}\label{eq:z_filter_global}
\hat{z}^g_{f_{\bar{p}}}(k)\doteq \frac{1}{2}\Big(z^{max}_g(k)+z^{min}_g(k) \Big),
\end{equation}
where
\begin{equation}\label{eq:zeta_i_global_minmax}
\begin{aligned}
z^{max}_g(k)&=\max_{z\in Z^g_{\bar{p}}(k)} z = \min_{p=1,\hdots,\bar{p}} \zeta_{g_p}^{max}, \\
z^{min}_g(k)&=\min_{z\in Z^g_{\bar{p}}(k)} z = \max_{p=1,\hdots,\bar{p}} \zeta_{g_p}^{min},
\end{aligned}
\end{equation}
and
\begin{equation*}
\begin{aligned}
\zeta_{g_p}^{max}&=\tilde{\varphi}_p(k-p)^T\hthetap+\taugph, \\[1\jot]
\zeta_{g_p}^{min}&=\tilde{\varphi}_p(k-p)^T\hthetap-\taugph.
\end{aligned}
\end{equation*}
The guaranteed accuracy bound of $\hat{z}^g_{f_{\bar{p}}}(k)$ is
\begin{equation}\label{eq:z_filt_glob_bound_new}
\tau^g_{f_{\bar{p}}}(k)=\frac{1}{2}\left\vert z_g^{max}(k)-z_g^{min}(k) \right\vert.
\end{equation}
The global filtering algorithm is sketched in Procedure \ref{p:MS_glob_filtering}.
\begin{procedure}
\caption{Multistep filtering with global bounds}
\label{p:MS_glob_filtering}
\begin{enumerate}
\item Perform the offline estimation of $\ulambp$, and define the FPSs, for $p\in[1,\,\bar{p}]$. Then, perform the offline identification of $\bar{p}$ independent $p$-steps ahead prediction models of the type of \eqref{eq:pred_formulation} by solving \eqref{eq:multistep_theta_LP} for each $p=1,\ldots,\bar{p}$.
\item At each time $k$, compute online the boundaries of the uncertainty interval $Z_{\bar{p}}^g(k)$ using \eqref{eq:zeta_i_global_minmax}.
\item Use the central algorithm \eqref{eq:z_filter_global} to obtain the filtered system output, and compute its guaranteed accuracy bound for the given time sample with \eqref{eq:z_filt_glob_bound_new}.
\end{enumerate}
\end{procedure}
\section{Simulation results}\label{s:sim_results}
The performance of the proposed filtering approaches have been tested on a numerical example, and compared to those of the well known Kalman filter, see e.g., \cite{kalman1961new,ljung1999system,soderstrom1989system} for details. We consider a sampled, single input single output linear time-invariant system in continuous time, whose transfer function is
\begin{equation*}\label{eq:num_sim_system}
G(s)=\frac{160}{(s+10)(s^2+0.8s+16)},
\end{equation*}
where $s$ is the Laplace variable. The input provided in the experiments is a three levels signal taking values in the set $\{-1, 0, 1\}$ randomly every 4 time instants. We collected 12000 input-output samples with a sampling time $T_s=0.1$.

\begin{figure}
	\centering
	\includegraphics[width=0.75\columnwidth]{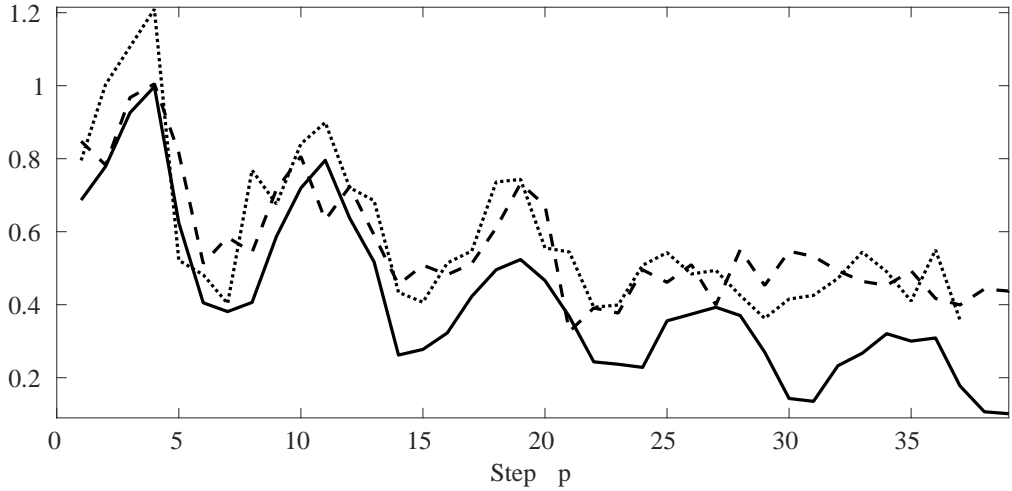}
	\caption{Multistep predictors guaranteed global error bound $\taugph$. Solid line: uniform measurement disturbance case; dashed line: Gaussian measurement disturbance case; dotted line: Gaussian process and measurement disturbances case.}
	\label{f:tau_p}
\end{figure}
\begin{table}[ht]
\caption{\label{t:Zfilt_uniform}Uniform measurement noise: Root Mean Square Error and accuracy bounds amplitude.}
\centering
\setlength\tabcolsep{7.5pt}
\begin{tabular}{c c c c c c c c c}
\toprule
\multicolumn{2}{c}{$\bar{p}$:} & $3$ & $5$ & $7$ & $15$ & $20$ & $35$ & 39 \\ 
\midrule
\multirow{2}{*}{local} & RMSE & 0.098 & 0.074 & 0.056 & 0.039 & 0.038 & 0.019 & 0.015 \\
 & $\max_k e(k)$ & 0.615 & 0.542 & 0.297 & 0.245 & 0.223 & 0.108 & 0.079 \\
\midrule
\multirow{2}{*}{global} & RMSE & 0.105 & 0.082 & 0.064 & 0.044 & 0.045 & 0.027 & 0.023 \\
 & $\max_k e(k)$ & 0.623 & 0.621 & 0.317 & 0.257 & 0.228 & 0.127 & 0.091 \\
\midrule
\multicolumn{3}{l}{Kalman filter - exact} & \multicolumn{3}{l}{RMSE = 0.001} & \multicolumn{3}{l}{$\max_k e(k)$ = 0.043}\\
\midrule
\multicolumn{3}{l}{Kalman filter - estimated} & \multicolumn{3}{l}{RMSE = 0.044} & \multicolumn{3}{l}{$\max_k e(k)$ = 0.228}\\
\midrule
\multirow{2}{*}{local} & \multirow{2}{*}{$\tau_{f_{\bar{p}}}$} avg & 0.344 & 0.288 & 0.184 & 0.117 & 0.110 & 0.054 & 0.040 \\
 & \phantom{$\tau_{f_{\bar{p}}}$} max & 0.658 & 0.565 & 0.327 & 0.252 & 0.242 & 0.119 & 0.087 \\
\midrule
\multirow{2}{*}{global} & \multirow{2}{*}{$\tau^g_{f_{\bar{p}}}$} avg & 0.633 & 0.563 & 0.371 & 0.252 & 0.251 & 0.128 & 0.095 \\
 & \phantom{$\tau^g_{f_{\bar{p}}}$} max & 0.687 & 0.624 & 0.381 & 0.262 & 0.262 & 0.135 & 0.101 \\
\midrule
\multicolumn{2}{c}{$\min\limits_{p\in[1,\bar{p}]} \taugph$} & 0.687 & 0.624 & 0.381 & 0.262 & 0.262 & 0.135 & 0.101 \\
\bottomrule
\end{tabular}
\vspace*{-0.2cm} 
\end{table}
We performed the numerical simulations considering three different disturbance scenarios: a) measurement disturbance signal given by a uniformly distributed random noise belonging to the interval $[-0.2, \; 0.2]$; b) measurement disturbance signal given by a Gaussian noise $d(t)\sim\mathcal{N}(0,\sigma_d^2)$, with $\sigma_d^2=0.01$; c) process and measurement disturbances given by the Gaussian noises $w(t)\sim\mathcal{N}(0,\sigma_w^2)$, and $d(t)\sim\mathcal{N}(0,\sigma_d^2)$ respectively, with $\sigma_w^2=0.001$, and $\sigma_d^2=0.01$. The measured output corresponds to $\tilde{y}(k)=z(k)+d(k)$, while, for scenario c), the system input is $u(k)=\tilde{u}(k)+w(k)$, where $\tilde{u}(k)$ is the measured input. The Kalman filter performance, used as a benchmark for comparison, are assessed under two different conditions: the first one is given by a Kalman filter based on a perfect knowledge of the system model, tuned to give the best filtering performance starting from unknown initial conditions, and using the noise affected measurements collected under scenarios a)-c); the second setup is given by a Kalman filter based on a model of the system identified from the noise affected measurements resorting to the simulation error method, see e.g. \cite{ljung1999system,tomita1992equation}, tuned to give the best filtering performance starting from unknown initial conditions. In both cases, we adopt the robust Kalman filter based on the Riccati equation approach.
\begin{figure*}[!htbp]
	\centering
		\begin{tabular}{cc}
				\hspace{-0.18cm}\begin{overpic}[width=0.48\columnwidth]{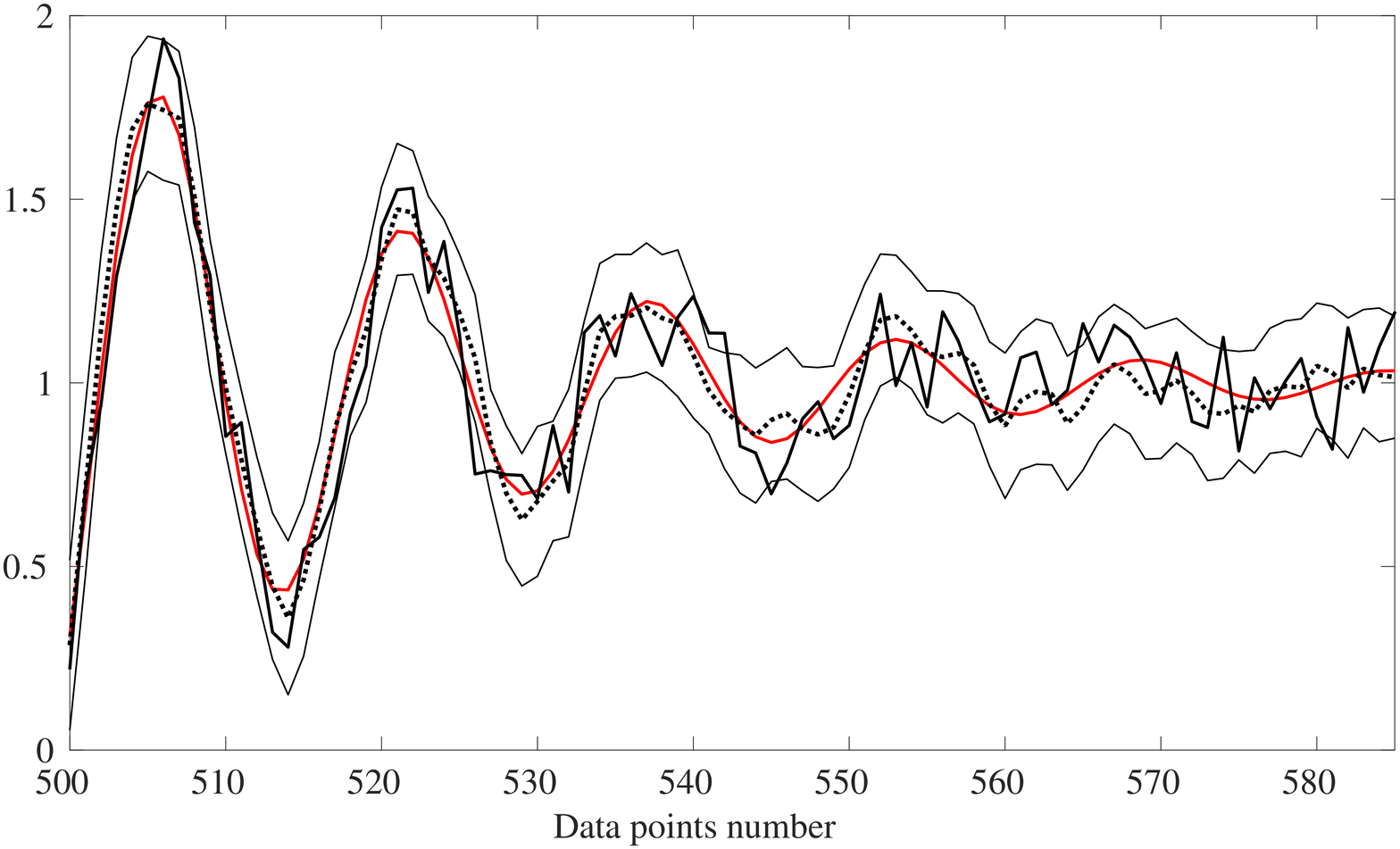}
					\put (47.8,61) {(a)}
				 \end{overpic} & \hspace{-0.2cm} \begin{overpic}[width=0.48\columnwidth]{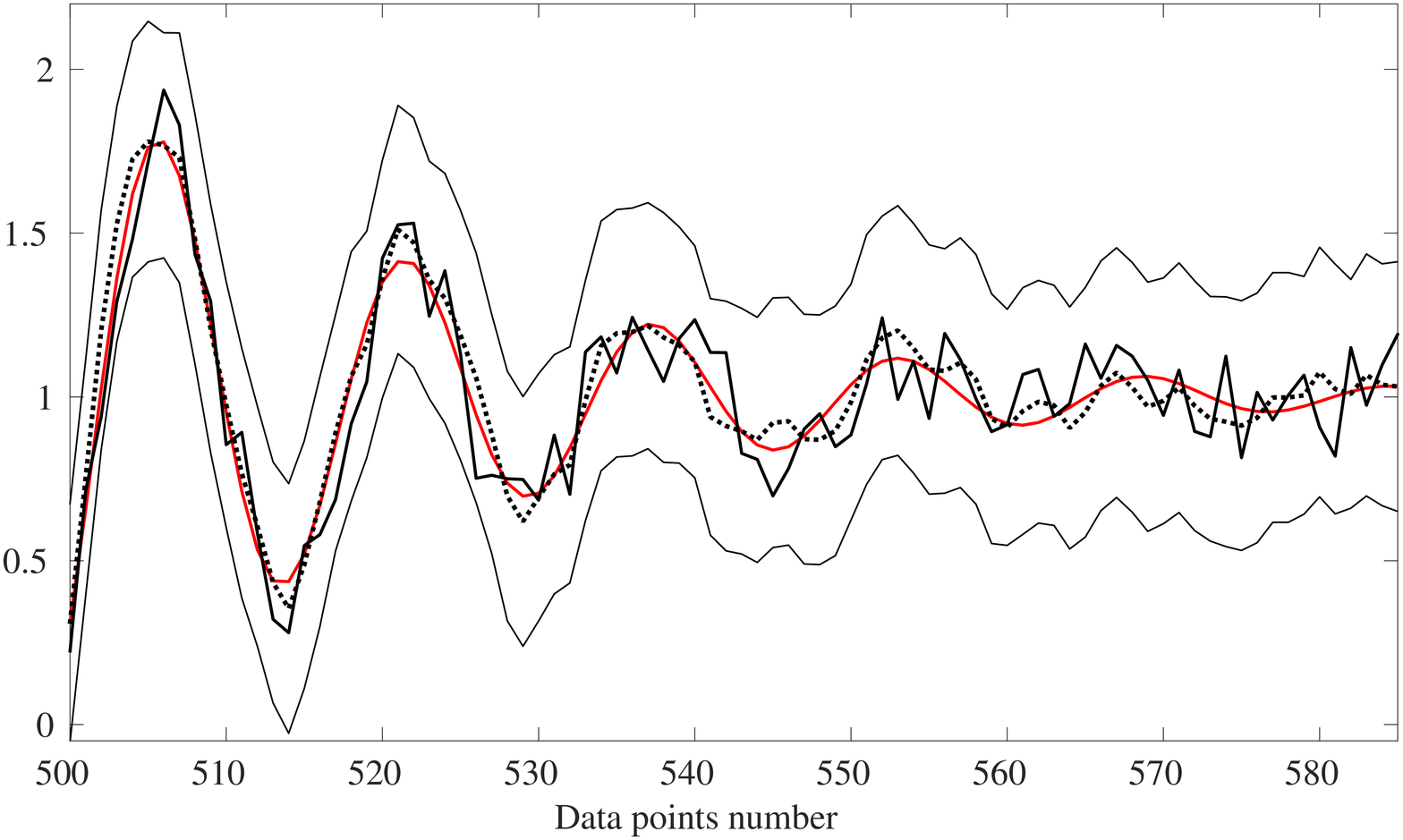}
					\put (47.8,61) {(b)}
				 \end{overpic}\\[4\jot]
     			 \hspace{-0.3cm} \begin{overpic}[width=0.48\columnwidth]{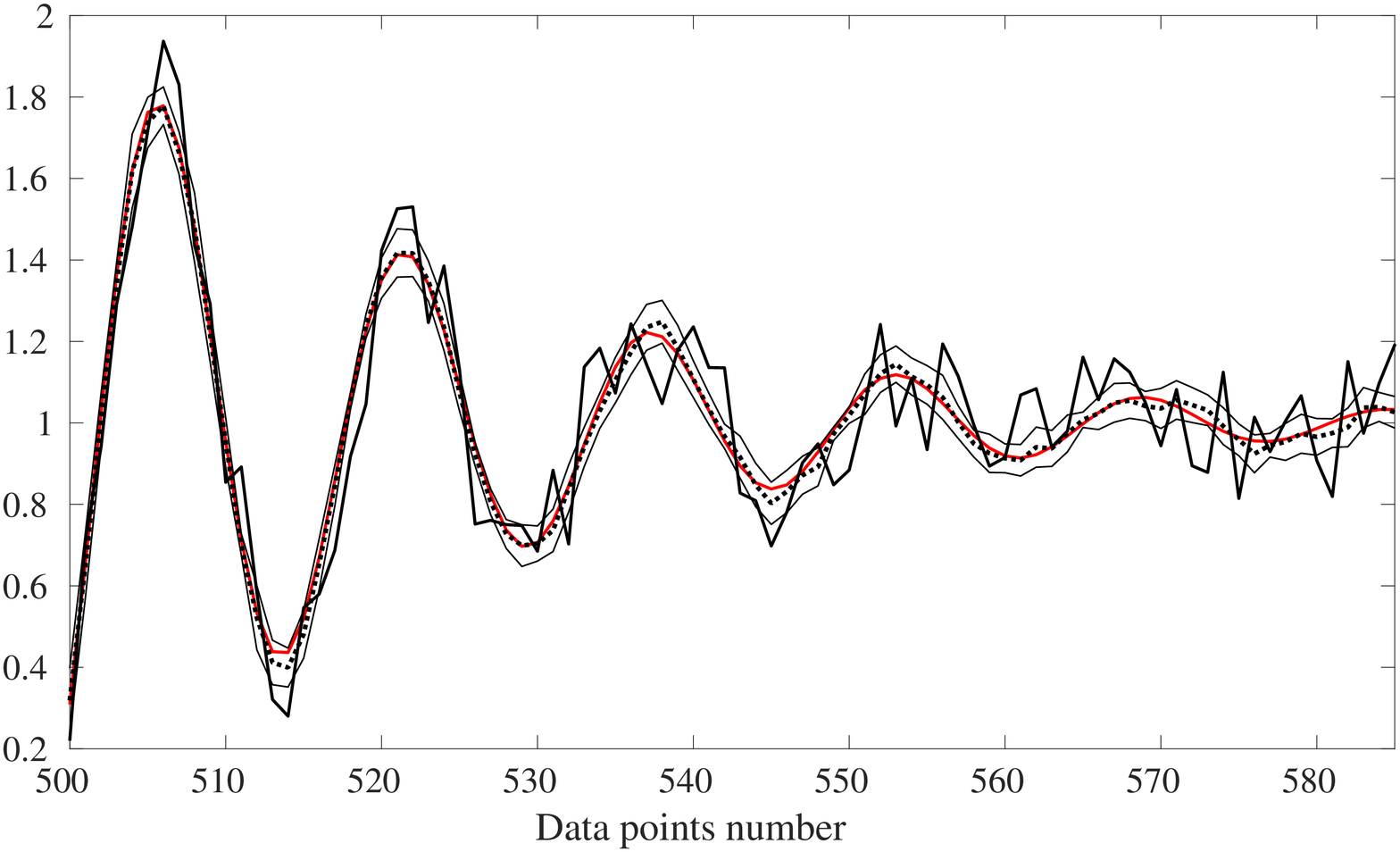}
					\put (47.8,61) {(c)}
				 \end{overpic} & \hspace{-0.2cm} \begin{overpic}[width=0.48\columnwidth]{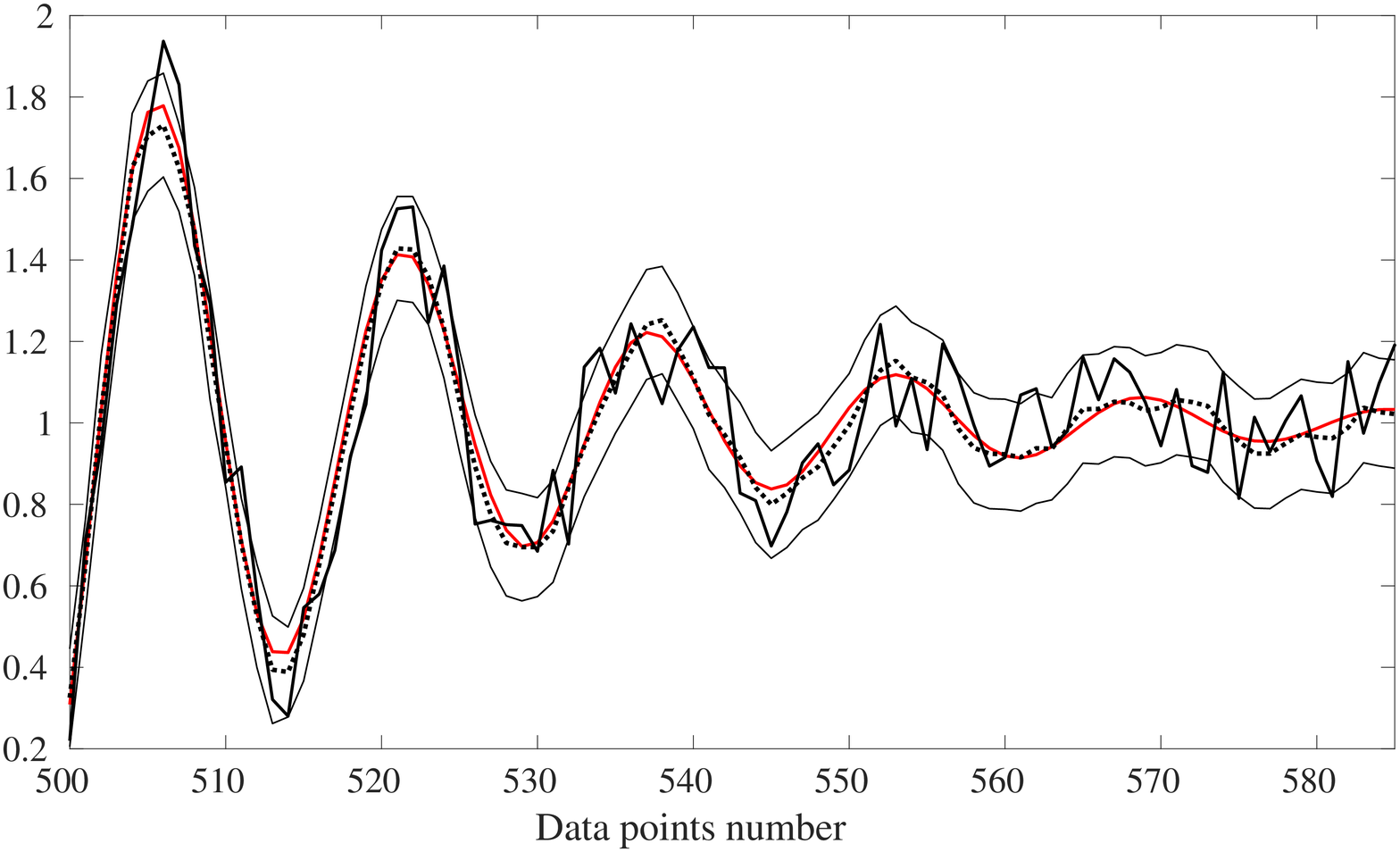}
					\put (47.8,61) {(d)}
				 \end{overpic}
		\end{tabular}
	\caption{Uniform measurement noise: filtered output $\hat{z}_{f_{\bar{p}}}$ with corresponding accuracy bounds. (a): local filtering with $\bar{p}=7$; (b): global filtering with $\bar{p}=7$; (c): local filtering with $\bar{p}=35$; (d): global filtering with $\bar{p}=35$. Black solid line: measured output $\tilde{y}$; red solid line: true system output $z$; black dotted line: filtered output $\hat{z}_{f_p}$; black thin lines: filtered output accuracy bounds.}
	\label{f:Zfilt_unif_loc_glob}
\end{figure*}
The first half of the available dataset is then used to perform the identification of the bound $\ulambp$, of the FPSs, of the multistep predictors, and of their guaranteed accuracy bounds $\taugph$, while the second half is used in validation to compute the filtered output using the local and the global filtering approaches, given by \eqref{eq:z_filter_local} and \eqref{eq:z_filter_global} respectively, and to obtain the corresponding uncertainty bounds \eqref{eq:z_filt_bound_new} and \eqref{eq:z_filt_glob_bound_new}. By applying the redundant constraints removal procedure, the number of FPS constraints is reduced from 12000 to an average of 331, for each $\Theta_p$, with $p\in[1,\,60]$, with a minimum of 48, and a maximum of 598 constraints. Fig.~\ref{f:tau_p} depicts the estimated global accuracy bounds $\taugph$ for the independent multistep prediction models, obtained with $\alpha=1.2$ and $\gamma=1.1$, for the three disturbance scenarios considered here.

Table~\ref{t:Zfilt_uniform} reports the guaranteed accuracy bounds, the Root Mean Squared Error (RMSE), defined as
\[
\text{RMSE}=\sqrt{\frac{\sum_{k=1}^N \Big( z(k)-\hat{z}_f(k) \Big)^2}{N}},
\]
and the maximum filtering error, given by $\max_k e(k)$, with $e(k)=z(k)-\hat{z}_f(k)$, for the Kalman filter, and for the local and global output filtering approaches, considering different values of the prediction horizon $\bar{p}$, for the case of uniformly distributed measurement noise. The accuracy bounds of the filtering approaches are then compared to that of the $p$-steps ahead model achieving the lowest $\taugph$ for each considered prediction horizon length. Tables~\ref{t:Zfilt_Gauss} and \ref{t:Zfilt_Gauss_IO} show the same performance comparison for the case of Gaussian measurement noise, and for the case of Gaussian process and measurement noises, respectively.
\begin{table}[ht]
\caption{\label{t:computation_time}Computational time comparison for local and global filtering approaches.}
\centering
\setlength\tabcolsep{10pt}
\vspace*{-0.2cm} 
\begin{tabular}{c c c c c}
\toprule
\multicolumn{2}{c}{$\bar{p}:$} & 3 & 8 & 15\\[-0.5\jot]
\midrule
\multirow{3}{*}{local} & min & 0.008 s & 0.022 s & 0.041 s \\[-0.3\jot]
& max & 0.033 s & 0.087 s & 0.137 s\\[-0.3\jot]
& avg & 0.009 s & 0.033 s & 0.046 s\\[-0.3\jot]
\midrule
\multirow{3}{*}{global} & min & 2.0$\cdot 10^{-5}$ s & 5.3$\cdot 10^{-5}$ s & 9.9$\cdot 10^{-5}$ s\\[-0.3\jot]
& max & 2.2$\cdot 10^{-4}$ s & 2.4$\cdot 10^{-4}$ s & 6.2$\cdot 10^{-4}$ s\\[-0.3\jot]
& avg & 2.8$\cdot 10^{-5}$ s & 8.7$\cdot 10^{-5}$ s & 1.2$\cdot 10^{-4}$ s\\[-0.3\jot]
\bottomrule
\end{tabular} 
\end{table}

These results show that the local filtering approach achieves the lowest filtering error, both in terms of RMSE and maximum error, and the lowest accuracy bounds, both on average and in worst-case, with respect to the global filtering approach, and to the usage of a single $p$-steps ahead prediction model. At the same time, the local filtering approach achieves good performances in terms of RMSE also when compared to the Kalman filter. It is worth to point out that, for $\bar{p}\geq 15$, the local filtering algorithm outperforms the Kalman filter based on an estimated system model, both in terms of RMSE and of maximum filtering error, for all the disturbance scenarios considered here. Moreover, for the case of process and measurement Gaussian disturbances, the local filtering approach is able to achieve a RMSE comparable to that of the Kalman filter based on the exact system model, and much smaller maximum error values. Clearly, in all the cases where a process disturbance signal is not present, the Kalman filter based on the exact system model achieves almost zero filtering error, since the input and the system model are perfectly known, and the only source of uncertainty is given by the unknown initial conditions.

The global filtering approach achieves intermediate results, closer to that obtained using a single multistep prediction model, with the advantage of a considerable reduction of the necessary online computational effort over the local approach. Table~\ref{t:computation_time} shows a brief comparison of the computational time required by the local and global approaches to provide the filtered version of one output sample. The simulations were performed on a Laptop equipped with Intel i7 dual-core processor with 2.4~GHz clock speed and 8~GB of RAM, using MatLab R2018b \verb|linprog| solver, based on the Dual Simplex algorithm. Fig.~\ref{f:Zfilt_unif_loc_glob} depicts the filtered output obtained using the local and global multistep filtering approaches, together with their accuracy bounds, for different prediction horizon lengths, and compare them to the real system output $z$, and to its measurement $\tilde{y}$. Fig.~\ref{f:MS_bound_on_z} depicts an example of the true output uncertainty intervals, obtained as intersection of the multistep predictors accuracy regions, for the case of uniform measurement noise. Finally, Fig.~\ref{f:Gauss_IO_Zfilt_vs_KF} reports the comparison between the filtered system output obtained using the local filtering approach, and that obtained from the Kalman filter based on the exact model of the system, and from the Kalman filter based on an identified system model, for the case of Gaussian process and measurement noise. 
\begin{figure}[!ht]
	\centering
	\includegraphics[width=0.75\columnwidth]{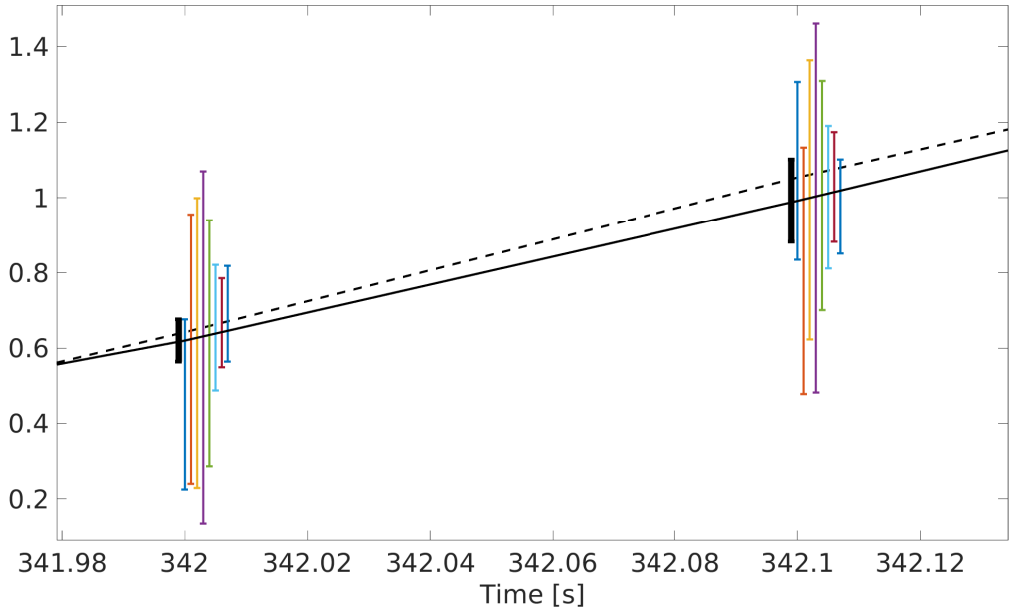}
	\caption{Uniform measurement noise: example of multistep predictors local accuracy intervals $[\zeta_p^{min}\;\zeta_p^{max}]$ and their intersection. Solid line: local filtered output $\hat{z}_{f_{\bar{p}}}$; dashed line: true system output $z$; colored bars: multistep predictors local accuracy intervals $[\zeta_p^{min},\;\zeta_p^{max}]$ for $\bar{p}=8$; thick black bars: resulting true output uncertainty set $Z_{\bar{p}}(k)$. \label{f:MS_bound_on_z}}
\end{figure}
\begin{table}[ht]
\caption{\label{t:Zfilt_Gauss}Gaussian measurement noise: Root Mean Square Error and accuracy bounds amplitude.}
\centering
\setlength\tabcolsep{7.5pt}
\begin{tabular}{c c c c c c c c c}
\toprule
\multicolumn{2}{c}{$\bar{p}$:} & $3$ & $5$ & $7$ & $15$ & $20$ & $35$ & 39 \\ 
\midrule
\multirow{2}{*}{local} & RMSE & 0.096 & 0.075 & 0.058 & 0.045 & 0.044 & 0.047 & 0.045 \\
 & $\max_k e(k)$ & 0.773 & 0.659 & 0.498 & 0.423 & 0.422 & 0.313 & 0.313 \\
\midrule
\multirow{2}{*}{global} & RMSE & 0.112 & 0.094 & 0.071 & 0.048 & 0.045 & 0.047 & 0.045 \\
 & $\max_k e(k)$ & 0.781 & 0.774 & 0.501 & 0.427 & 0.426 & 0.325 & 0.325 \\
\midrule
\multicolumn{3}{l}{Kalman filter - exact} & \multicolumn{3}{l}{RMSE = 0.002} & \multicolumn{3}{l}{$\max_k e(k)$ = 0.093}\\
\midrule
\multicolumn{3}{l}{Kalman filter - estimated} & \multicolumn{3}{l}{RMSE = 0.053} & \multicolumn{3}{l}{$\max_k e(k)$ = 1.089}\\
\midrule
\multirow{2}{*}{local} & \multirow{2}{*}{$\tau_{f_{\bar{p}}}$} avg & 0.355 & 0.317 & 0.212 & 0.164 & 0.156 & 0.110 & 0.107 \\
 & \phantom{$\tau_{f_{\bar{p}}}$} max & 0.744 & 0.686 & 0.470 & 0.432 & 0.422 & 0.295 & 0.295 \\
\midrule
\multirow{2}{*}{global} & \multirow{2}{*}{$\tau^g_{f_{\bar{p}}}$} avg & 0.732 & 0.704 & 0.509 & 0.433 & 0.421 & 0.309 & 0.307 \\
 & \phantom{$\tau^g_{f_{\bar{p}}}$} max & 0.783 & 0.783 & 0.513 & 0.455 & 0.455 & 0.326 & 0.326 \\
\midrule
\multicolumn{2}{c}{$\min\limits_{p\in[1,\bar{p}]} \taugph$} & 0.783 & 0.783 & 0.513 & 0.455 & 0.455 & 0.326 & 0.326 \\
\bottomrule
\end{tabular} 
\end{table}
\begin{table}[ht]
\caption{\label{t:Zfilt_Gauss_IO}Gaussian process and measurement noise: Root Mean Square Error and accuracy bounds amplitude.}
\centering
\setlength\tabcolsep{7.5pt}
\begin{tabular}{c c c c c c c c}
\toprule
\multicolumn{2}{c}{$\bar{p}$:} & $3$ & $5$ & $7$ & $15$ & $25$ & $35$ \\ 
\midrule
\multirow{2}{*}{local} & RMSE & 0.095 & 0.085 & 0.058 & 0.043 & 0.037 & 0.037 \\
 & $\max_k e(k)$ & 0.379 & 0.337 & 0.231 & 0.206 & 0.203 & 0.165 \\
\midrule
\multirow{2}{*}{global} & RMSE & 0.108 & 0.100 & 0.066 & 0.042 & 0.039 & 0.037 \\
 & $\max_k e(k)$ & 0.395 & 0.463 & 0.360 & 0.326 & 0.308 & 0.185 \\
\midrule
\multicolumn{3}{l}{Kalman filter - exact} & \multicolumn{2}{l}{RMSE = 0.030} & \multicolumn{3}{l}{$\max_k e(k)$ = 0.884}\\
\midrule
\multicolumn{3}{l}{Kalman filter - estimated} & \multicolumn{2}{l}{RMSE = 0.043} & \multicolumn{3}{l}{$\max_k e(k)$ = 0.933}\\
\midrule
\multirow{2}{*}{local} & \multirow{2}{*}{$\tau_{f_{\bar{p}}}$} avg & 0.392 & 0.253 & 0.165 & 0.135 & 0.111 & 0.097 \\
 & \phantom{$\tau_{f_{\bar{p}}}$} max & 0.730 & 0.514 & 0.336 & 0.328 & 0.298 & 0.278 \\
\midrule
\multirow{2}{*}{global} & \multirow{2}{*}{$\tau^g_{f_{\bar{p}}}$} avg & 0.778 & 0.519 & 0.397 & 0.367 & 0.349 & 0.320 \\
 & \phantom{$\tau^g_{f_{\bar{p}}}$} max & 0.795 & 0.522 & 0.403 & 0.403 & 0.393 & 0.363 \\
\midrule
\multicolumn{2}{c}{$\min\limits_{p\in[1,\bar{p}]} \taugph$} & 0.795 & 0.522 & 0.403 & 0.403 & 0.393 & 0.363 \\
\bottomrule
\end{tabular} 
\end{table}
\begin{figure*}[!ht]
	\centering
	\begin{tabular}{c c}
	(a) & (b)\\
	\includegraphics[width=0.48\columnwidth]{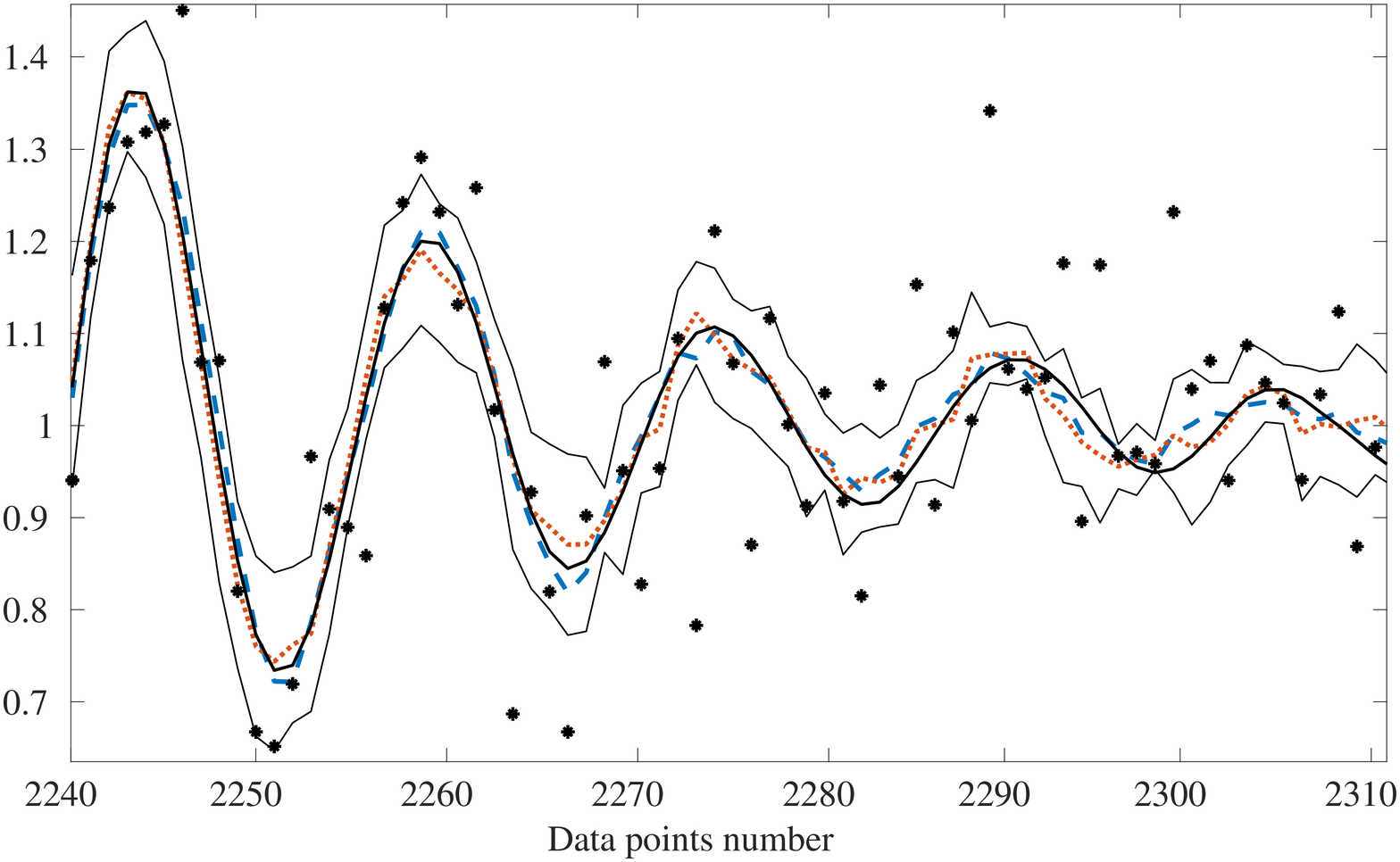} & \includegraphics[width=0.48\columnwidth]{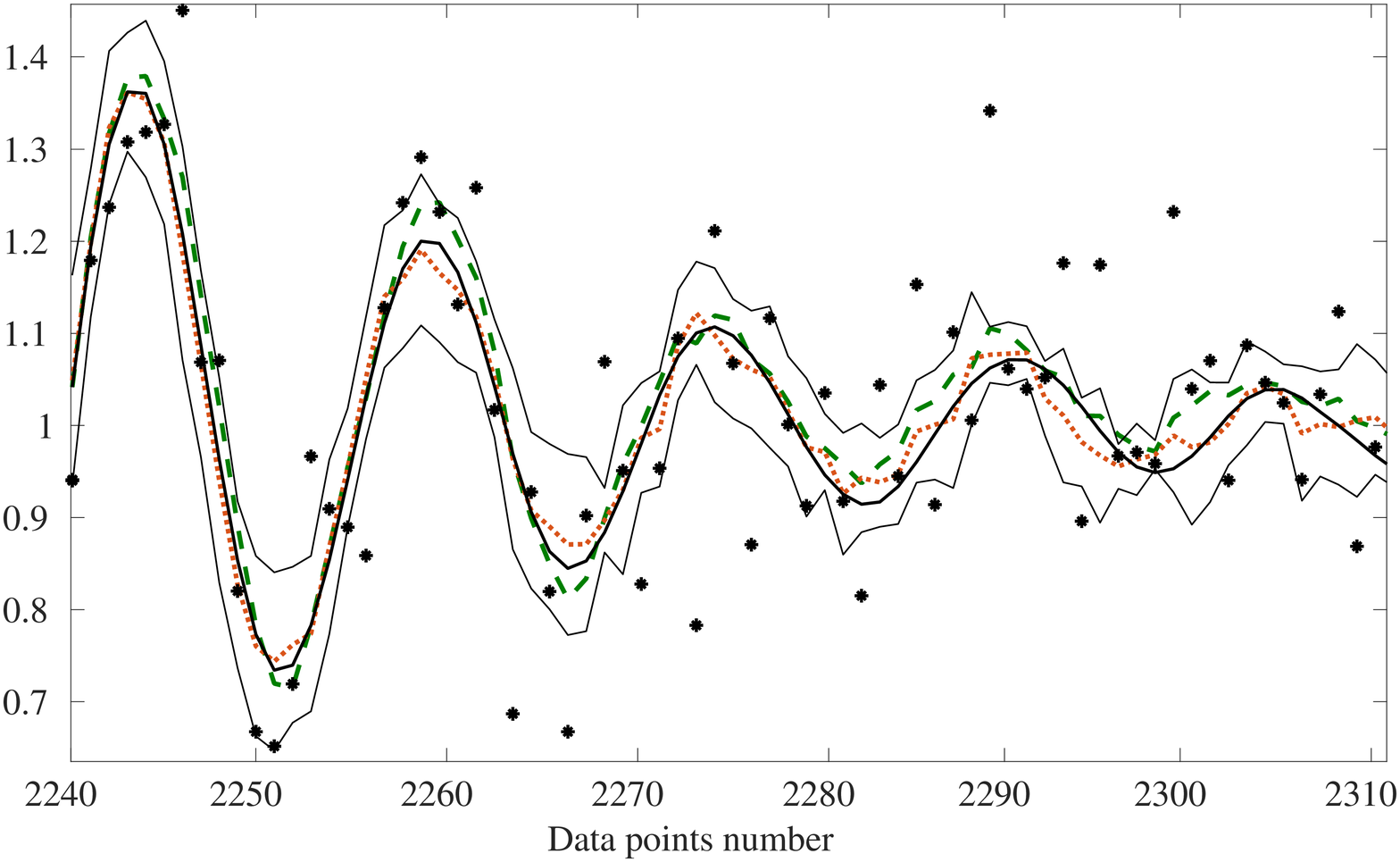}
	\end{tabular}
	\caption{Gaussian process and measurement noise: comparison of local filtering algorithm and Kalman filter. (a): local filtering approach and Kalman filter with exact system model; (b): local filtering approach and Kalman filter with estimated system model. Black solid line: true system output $z$; dots: measured output $\tilde{y}$; red dotted line: local filtered output $\hat{z}_{f_{\bar{p}}}$, with $\bar{p}=35$; black thin lines: local filtered output accuracy bounds; blue dashed line: Kalman filter with exact system model; green dashed line: Kalman filter with estimated system model.\label{f:Gauss_IO_Zfilt_vs_KF}}
\end{figure*}
\section{Conclusions}\label{s:conclusions}
The data-driven direct filtering approach presented in this paper allows one to address the problem of filtering the output of linear time-invariant systems subject to unknown-but-bounded uncertainties without deriving a model of the system. The proposed filtering algorithm is able to achieve good filtering accuracy, quantified in terms of average filtering error, and of guaranteed accuracy bounds, and it allows one to compute tight guaranteed uncertainty intervals for the  true system output. The proposed method relies on the intersection of the uncertainty regions of different SM multistep predictors, which are used together to provide a filtered version of the system output, while refining the output uncertainty intervals thanks to the combined use of different prediction horizon lengths. The local filtering algorithm, and its accuracy bounds, are then obtained by means of linear programming, which makes them suitable for online implementation in real world applications. A more conservative global filtering approach was presented to lower the online computational effort, moving the solution of the optimization problems to an offline phase. Numerical simulations illustrate the performance of the proposed local filtering approach, which is able to outperform a Kalman filter based on an estimated system model, and to score similar performance to that of a Kalman filter based on the exact system model, while retaining the ability to provide optimal guaranteed accuracy bounds for the filtered output.
\bibliographystyle{IEEEtran}

\vfill
\end{document}